\documentclass[11pt]{article}

\usepackage{amsmath}
\usepackage{amsfonts}
\usepackage{color}
\usepackage{amssymb}
\usepackage{hyperref}
\usepackage{graphicx}
\usepackage{enumerate}
\usepackage[all]{xy}

 \allowdisplaybreaks

\begin{document}

\newtheorem{problem}{Problem}

\newtheorem{theorem}{Theorem}[section]
\newtheorem{corollary}[theorem]{Corollary}
\newtheorem{definition}[theorem]{Definition}
\newtheorem{conjecture}[theorem]{Conjecture}
\newtheorem{question}[theorem]{Question}
\newtheorem{lemma}[theorem]{Lemma}
\newtheorem{proposition}[theorem]{Proposition}
\newtheorem{quest}[theorem]{Question}
\newtheorem{example}[theorem]{Example}
\newenvironment{proof}{\noindent {\bf
Proof.}}{\rule{2mm}{2mm}\par\medskip}
\newenvironment{proofof}{\noindent {\bf
Proof of the Theorem 6.1.}}{\rule{2mm}{2mm}\par\medskip}
\newcommand{\remark}{\medskip\par\noindent {\bf Remark.~~}}
\newcommand{\pp}{{\it p.}}
\newcommand{\de}{\em}

\title{  {A tight $Q$-index condition for a graph to
be $k$-path-coverable involving minimum degree}\thanks{This research was supported by  NSFC (Nos. 11871479, 12071484), Hunan Provincial Natural Science Foundation (2020JJ4675, 2018JJ2479).
 E-mail addresses: taocheng@sdnu.edu.cn(Tao Cheng), fenglh@163.com(L. Feng), ytli0921@hnu.edu.cn(Y. Li), wjliu6210@126.com(W. Liu).}}

\date{April 26, 2020}

\author{Tao Cheng$^a$, Lihua Feng$^b$,  Yongtao Li$^c$, Weijun Liu$^{b}$\\
{\small $^a$ School of Mathematics and Statistics, Shandong Normal University } \\
{\small  Jinan, Shandong, 250014, P.R. China. } \\
{\small $^b$School of Mathematics and Statistics, Central South University} \\
{\small New Campus, Changsha, Hunan, 410083, P.R. China. } \\ 
{\small ${}^c$School of Mathematics, Hunan University} \\
{\small Changsha, Hunan, 410082, P.R. China } \\
 }

\maketitle

\vspace{-0.5cm}

\begin{abstract}
A graph $G$ is $k$-path-coverable if its vertex set $V(G)$ can be covered by $k$ or fewer vertex disjoint paths.  In this paper, using the $Q$-index of a connected graph $G$, we present a tight sufficient  condition for $G$  with fixed minimum degree and  large order to be    $k$-path-coverable.
 \end{abstract}


\section{Introduction}
 We only consider  simple connected graphs throughout  paper.  The notations we used are standard from \cite{BondyBook}. Let $G$ be a simple connected  graph with vertex set $V(G)$ and edge set $E(G)$ such that $|V(G)|=n$ and $|E(G)|=e(G)$. We use $d(v)$ to denote the degree of a vertex $v$ in $G$, and  the minimum degree is   $\delta(G)=\delta$. For two vertex-disjoint graphs $G$ and $H,$   $G\cup H$ is denoted to be the disjoint union of $G$ and $H,$ $G\vee H$ the join of  $G$ and $H,$ which is obtained from $G\cup H$ by adding all possible edges between $G$ and $H.$
We use the symbol $i\sim j$ to denote the vertices $i$ and $j$ are adjacent, and $i\nsim j$ otherwise.

 The \emph{adjacency matrix} of $G$ is $A(G)=(a_{ij})_{n \times n}$ with $a_{ij}=1$ if  $i$ and $j$ are adjacent, and $a_{ij}=0$ otherwise.  The largest eigenvalue of $A(G)$, denoted by $\lambda (G)$, is called the {\it spectral radius} of $G$. The \emph{diagonal matrix} of $G$ is $D(G)=(d_{ii})_{n \times n}$ with diagonal entry  $d_{ii}=d(i)$. The \emph{signless Laplacian matrix} $Q(G)$ of $G$ is   $D(G)+A(G).$ The largest eigenvalue of $Q(G)$, denoted by $q (G)$, is called the  $Q$-\textit{index} (or the signless Laplacian  spectral radius) of $G$.

For  people that working on spectral graph theory,   one of the most well-known problems  is the
Brualdi-Solheid problem \cite{brualdi}:  Given a set ${\cal
{G}}$ of graphs, find a tight upper bound for the spectral radius in
${\cal
{G}}$ and characterize the extremal graphs.
This problem is well studied in the
literature for many classes of graphs, such as  graphs with given diameter \cite{Hansen}, edge chromatic number \cite{FengLAA16}, domination number
\cite{Stevanovic}.
   For the  $Q$-index counterpart of the above problem, Zhang \cite{Zhangxiaodong} gave  the  $Q$-index of graphs with given degree sequence, Zhou \cite{Zhoubo} studied the $Q$-index and Hamiltonicity.  Also, from both theoretical and practical viewpoint,
the  eigenvalues of graphs have been successfully used in many   other disciplines,  one may
refer to \cite{Huobofeng, LiShiEnergy, Lixueliang13, ZhangMinjieDAM, ZhangMinjieAMC}.

Analogous to the Brualdi--Solheid problem, the following    problem regarding the adjacency  spectral radius
was proposed in \cite{NikiforovLAA10}:
What is the maximum spectral radius of a graph $G$ on $n$
vertices without a subgraph isomorphic to a given graph $F$?
 For this problem,
 Fiedler and Nikiforov  \cite{FiedlerNikif} obtained tight sufficient conditions for
graphs to be hamiltonian or traceable. This motivates   further study for such question,
 see   \cite{FengDAM17, FengMonoshMath, FiedlerNikif,   ZhouWangligongAMC18, Liyawei, LiuDMGT, Lumei, Zhoubo,  Liuruifang, ZhouWangligong, Ningbo15}.

When the minimum degree is involved in contrast with the results in  \cite{FiedlerNikif}, using the adjacency spectral radius, Li and Ning \cite{LiBinlongNingbo}   obtained
\begin{theorem}\label{LiBinlongNingbo}\cite{LiBinlongNingbo}
Let $t \geq 1$, and $G$ be a graph of order $n$ with minimum degree $\delta(G) \geq t$.
 If $n \geq \max\{6t+10, (t^2+7t+8)/2\}$ and
$$\lambda(G) \geq \lambda(K_t \vee (K_{n-2t-1} \cup (t+1)K_1)),$$
then $G$ is traceable, unless $G=K_t \vee (K_{n-2t-1} \cup (t+1)K_1)$.
\end{theorem}

Theorem \ref{LiBinlongNingbo} is   generalized by Nikiforov  as
\begin{theorem}\cite{Nikiforov} \label{Nikiforov}
Let $t \geq 1$  and $G$ be a graph of order $n$.
If $n \geq t^3 +t^2+2t + 5$, $\delta(G) \geq t$
and
 $$
\lambda(G) \geq n-t-2,
$$ then $G$ is traceable unless $G = K_{t+1}\cup K_{n-t-1} $ or $G = K_t \vee (K_{n-2t-1} \cup (t+1)K_1)$.
\end{theorem}

In this paper, we  consider the $k$-path-coverable problem.
A graph
$G$ is {\it $k$-path-coverable} if $V(G)$ can be covered by $k$ or fewer vertex-disjoint paths. In particular, 1-path-coverable is the same as  traceable.
The disjoint path cover problem is strongly related to the well-known hamiltonian problem (one may refer to \cite{Lihao} for a survey), which is among the most fundamental
ones in graph theory, and   attracts much attention in theoretical computer science. However, this problem is NP-complete \cite{Steiner}, therefore, finding their guaranteed sufficient conditions becomes an interesting work. In \cite{LiJianping}, such Ore-type condition is obtained.

%

For convenience,  we denote
$$
{B}(n,k,\delta):=K_{\delta}\vee ( K_{n-2\delta -k} \cup \overline{K_{\delta +k}}).
$$
In \cite{LiuDMGT}, by generalizing the results in Theorem  \ref{LiBinlongNingbo} and Theorem \ref{Nikiforov}, Liu et. al. obtained the following sufficient
by using the adjacency spectral radius.
 \begin{theorem} \label{thm13}\cite{LiuDMGT}
Let $k\ge 1$ and $\delta \ge 2$. If $G$ is a connected graph on $n \ge
\max \{\delta^2(\delta +k)+\delta +k+5, 5k+6\delta +6\}$ vertices and minimum degree
$\delta (G)\ge \delta $ such that
\[ \lambda (G) \ge n-\delta -k -1, \]
then $G$ is $k$-path-coverable unless $G=B(n,k,\delta )$.
\end{theorem}

In this paper, we will consider the $Q$-index version  of Theorem \ref{thm13}.
For the graph ${B}(n,k,\delta)$,
 let
$$X=\{v \in V({B}(n,k,\delta)):d(v)=k\}, \qquad
Y=\{v \in V({B}(n,k,\delta)):d(v)=n-1 \},$$
$$Z=\{v \in V({B}(n,k,\delta)):d(v)=n-k-\delta -1 \}.$$
We denote by $E_1$ the edge set of $E({R}(n,k,\delta))$ whose endpoints are both from $Y\cup Z.$ We define
\begin{gather*}
 \mathcal{B}_1(n,k,\delta)=\left\{G\subseteq {R}(n,k,\delta)-E',
\mbox{where $ E'\subset E_1$ with $|E'|\leq
\left\lfloor \frac{ (\delta +k)\delta }{4}\right\rfloor$}\right\},\\
\mathcal{B}_2(n,k,\delta)=\left\{G\subseteq {R}(n,k,\delta)-E',
\mbox{where $ E'\subset E_1$ with $|E'|=
\left\lfloor \frac{ (\delta +k)\delta}{4}\right\rfloor+1 $}\right\}.
\end{gather*}
Here, the symbol $G\subseteq {R}(n,k,\delta)-E'$ means $G$ is a subgraph of ${R}(n,k,\delta)-E'$.

The main result of this paper is
\begin{theorem} \label{thm14}
Let $k\ge 1$ and $\delta \ge 2$. If $G$ is a connected graph on $n \ge
(\delta^2 +k\delta +7\delta +6k +9)(\delta^2 +k\delta +1)$ vertices and minimum degree
$\delta (G)\ge \delta $ such that
\[ q(G) \ge 2(n-\delta -k -1), \]
then $G$ is $k$-path-coverable unless $G\in \mathcal{B}_1(n,k,\delta )$.
\end{theorem}

\section{Preliminary Results}
We first present some basic notations and results.
For $x=(x_{1},x_{2},...,x_{n})^T\neq 0$, according to the Rayleigh's principle, we have
\[ q(G)=\max_{x}\frac{\langle Q(G)x,x \rangle}{\langle x,x \rangle}
=\max_{x}\frac{ x^{T}Q(G)x }{x^{T}x}.\]
By the definition of $Q(G),$ we   have
\[  \langle Q(G)x,x\rangle=\sum_{i\in V(G)}d(i)x_{i}^2+ 2\sum_{i\sim j}x_{i}x_{j}
=\sum_{i\sim j}(x_{i}+x_{j})^2.\]
If ${h}$ is the unit positive eigenvector of $q(G)$,
then
\[ Q(G)h=q(G)h.\]
If we take the $i$-th entry of both sides and rearranging terms, we get
\begin{eqnarray} \label{eq1}
(q(G)-d(i))h_{i}=\sum_{i\sim j}h_{j}.
\end{eqnarray}
Let $N(i)$ denote the neighbor of $i,$ $N[i]=N(i)\cup \{i\}$. From the above,  we have
\begin{lemma}\label{lem21} \cite{Liyawei}
For any $i,j\in V(G),$ we have
\begin{eqnarray} \label{eq2}
   (q(G)-d(i))(h_{i}-h_{j})
=(d(i)-d(j))h_{j}+\sum_{k\in N(i)\setminus N(j)}h_{k}-\sum_{l\in N(j)\setminus N(i)}h_{l}.
\end{eqnarray}
\end{lemma}

We now present some graph notations we will use.
The concept of stability was introduced by Bondy and Chv\'{a}tal \cite{Bondy}.
Let $P$ be a property defined on all graphs of order $n$.
Let $k$ be a non-negative integer. The property $P$ is said to be $k$-stable
if whenever $G+uv$ has property $P$ and
\[ d_G(u)+d_G(v) \ge k, \]
where $un \notin E(G)$, then $G$ itself has property $P$.
Among all the graphs $H$ of order $n$ such that $G\subseteq H$ and
\[ d_H(u)+d_H(v) <k \]
for all $uv\notin E(H)$, there is a unique smallest on with respect to size. We shall call this graph
the $k$-\textit{closure} of G, and denote it by $\mathrm{cl}_k(G)$. Obviously, $\mathrm{cl}_k(G)$
can be obtained from $G$ by recursive procedure which consists of  joining non-adjacent vertices
with degree-sum at least $k$. This concept plays a prominent role in any structual graph theory problems.

\begin{lemma} \cite{FengPIMB09} \label{lem22}
Let $G$ be a graph of order $n.$ Then
  $$ q(G)\leq \frac{2e(G)}{n-1}+n-2.$$
\end{lemma}

\begin{lemma}\cite{Bondy} \label{lem23}
The property that `` $G$ is $k$-path-coverable'' is $(n-k)$-stable.
\end{lemma}

\section{Proof of the Main Results}
In order to prove our main result,   we  need the following stability result.

\begin{lemma} \label{lem31}
Let $G$ be a graph of order $n\ge 8\delta +6k +7$, where $k\ge 1$.
If $\delta (G) \ge \delta \ge 1$ and
\[ e(G)>\frac{1}{2}(n-k-\delta -1)(n-k-\delta -2) +(\delta +k+1)(\delta +1), \]
then $G$ is $k$-path-coverable unless $G\subseteq B(n,k,\delta )$.
\end{lemma}

\begin{proof}
Let $G'=\mathrm{cl}_{n-k}(G)$. If $G'$ is $k$-path-coverable, then so is $G$
by the $(n-k)$-stablity of $k$-path-coverable. We now assume that $G'$ is not $k$-path-coverable,
note that $\delta (G') \ge \delta (G)\ge \delta$ and $e(G')\ge e(G)$. Next, we show that
\[ \omega (G') =n-\delta -k . \]
Let $C$ be the vertex set of a maximum clique of $G' $, since every two vertices whose degree sum is
at least $n-k$ are adjacent, then $C$ contains all vertices whose degree is at least $(n-k)/2$.
Let $H=G'-C$ and $t=|C|$.

Suppose first that $1\le t \le (n+k)/3+\delta +4/3$, observe that $d_{G'}(v) \le (n-k-1)/2$
for every $v\in V(H)$, if otherwise, we may assume $d_{G'}(v)>(n-k-1)/2$,
because $d_{G'}(v)$ is a positive integer, then $d_{G'}(v)\ge (n-k-1)/2+1/2=(n-k)/2$,
so $v$ is contained in $C$, a contradiction. Clearly, $d_C(v)\le t-1$,
so some tedious manipulation yields that
\begin{align*}
e(G')
&=e(G'[C]) +e(H) +e(V(H),C) \\
&=e(G'[C]) +\frac{1}{2}\left(  \sum\limits_{v\in V(H)} d_{G'}(v) +
\sum\limits_{v\in V(H)} d_C(v)\right) \\
&\le {t \choose 2} +\frac{1}{2}(n-t)\left( \frac{n-k-1}{2} +t-1\right) \\
&=\frac{n+k+1}{4}t +\frac{n(n-k-3)}{4} \\
&\le \frac{n+k+1}{4}\left( \frac{n+k}{3}+\delta +\frac{4}{3}\right) +\frac{n(n-k-3)}{4}  \\
&=\frac{1}{3}n^2 +\frac{3\delta -k -4}{12}n +(k+1)(3\delta -k +3) \\
&\le \frac{1}{2}(n-k-\delta -1)(n-k-\delta -2) +(\delta +1)(\delta +k+1),
\end{align*}
which leads to a contradiction (The last inequality follows due to $n\ge 8\delta +6k +7$).

Suppose second that $(n+k)/3+\delta +4/3 <t \le n-\delta -k -1$,
note that $d_{G'}(v)\le n-k-t$ for every $v\in V(H)$,
if otherwise, we assume $d_{G'}(v)\ge n-k-t+1$, then $v$ will be adjacent to every vertex in $C$.
Therefore
\begin{align*}
e(G')&= e(G'[C]) +e(H) +e(V(H),C) \\
&\le e(G'[C]) +\sum\limits_{v\in V(H)} d_{G'}(v) \\
&\le {t \choose 2} +(n-t)(n-k-t) \\
&=\frac{3}{2}t^2+ (-2n +k-\frac{1}{2})t +n(n-k) \\
&\le \frac{3}{2}\left( n-\delta -k-1\right)^2 +(-2n +k-\frac{1}{2})(n-\delta -k-1)+n(n-k)\\
&= \frac{1}{2}(n-k-\delta -1)(n-k-\delta -2) +(\delta +1)(\delta +k+1),
\end{align*}
also a contradiction.

Finally, suppose that $t\ge n-\delta -k +1$, since $G' $ is not a clique, then $V(H)$ is not empty.
Note that $d_{G'}(v) \ge \delta $ for every $v \in V(H)$ and $d_{G'}(u) \ge t-1\ge n-k-\delta $
for every $u\in C$, this means that every vertex of $V(H)$ is adjacent to every vertex of $C$,
this contradicts the fact that $C$ is a maximum clique.
Consequently, we now infer that $\omega (G')=n-\delta -k$.

We call that a vertex in $C$ is a {\it frontier} vertex if it has degree at least $n-\delta -k$ in $G'$,
that is to say, it has at least one neighbour in $H$, and denote $F=\{u_1,u_2,\ldots ,u_s\}$ by
the set of frontier vertices. Since $\omega (G')=n-\delta -k$, then $d_{G'}(v)\ge n-\delta -k -1$
for every $v\in C$, we can see by the $(n-k)$-stablity that every vertex in $H$ has degree exactly
$\delta $ in $G'$, furthermore, every vertex in $H$ is adjacent to every frontier vertex in $G'$, and
then we have $1\le s\le \delta$. Since $C$ is a clique $K_{n-\delta -k}$,
then we can choose a path $P$ in $C-F$ with two end-vertices $u_1$ and $u_s$.

We assert that $s=\delta$.

If $s=1$, then every vertex in $H$ has degree $\delta -1$ in $H$.
According to Dirac's theorem,
$H$ has a path of order at least $\delta $.
First, we assume that $H$ has a path $P'$ of order $\delta +1$,
and let $x,x'$ be the two end-vertices of $P'$, then $G'$ can be covered by one path
$x'P'xu_1P$ and the rest $(k-1)$ vertices in $H$, so $G'$ is $k$-path-coverable, a contradiction.
We now assume that $H$ has no paths of order more than $\delta $,
let $P'$ be a path of order $\delta$ in $H$ and $V(H-P')=\{v_1,v_2,\ldots ,v_{k}\}$.
Observe that $x$ has no neighbour in $V(H-P')$ and $d_{H}(x)=\delta -1$, which implies that
$xx'\in E(H)$. Since $H$ has no paths of order more than $\delta$, then every vertex in $V(H-P')$
has no neighbour in $P'$, so $H-P'$ has an edge, say, $v_1v_2$. Then $G'$ can be covered by two paths
$x'P'xu_1P,v_1v_2$ and the rest $(k-2)$ vertices in $H$, Also a contradiction.

If $2\le s\le \delta -1$, so every vertex in $H$ has degree $\delta -s$ in $H$. By Dirac's theorem,
$H$ has a path of order at least $\delta -s+1$. We first assume that $H$ has a path $P'$ of order
$\delta -s +2$ and $V(H-P')=\{v_1,v_2,\ldots ,v_{s+k-2}\}$, then
$v_1u_1v_2u_2\cdots v_{s-1}u_{s-1} x'P'x u_sP$ is path of order $n-(k-1)$, so $G'$ is
$k$-path-coverable, a contradiction. Now we assume that $H$ has no paths of order more than
$\delta -s +1$, let $P'$ be a path of order $\delta -s +1$ in $H$ and $V(H-P')=\{v_1,v_2,
\ldots ,v_{k+s-1}\}$.
Observe that $x$ has no neighbour in $V(H-P')$ and $d_{H}(x)=\delta -s$, which implies that
$xx'\in E(H)$. Since $H$ has no paths of order more than $\delta -s-1$, then every vertex in $V(H-P')$
has no neighbour in $P'$, so $H-P'$ has an edge, say, $v_1v_2$. Then
$v_su_1v_1v_2u_2v_3u_3\cdots v_{s-1}u_{s-1} x'P'xu_sP$ is a path of order $n-(k-1)$,
so $G'$ is $k$-path-coverable, this also leads to a contradiction.

We have thus proved that $s=\delta $, as we claimed.
Then $H$ is an independent vertices $\overline{K_{\delta +k}}$, i.e., $G'=B(n,k,\delta)$, we have $G\subseteq B(n,k,\delta)$. The proof  is now complete.
\end{proof}

\begin{lemma}\label{lem32}
For each graph $G\in \mathcal{B}_1(n,k,\delta)$,
we have $q(G)\geq 2(n-k-\delta -1)$.
\end{lemma}

\begin{proof}
Let $G\in \mathcal{B}_1(n,k,\delta ).$
Note that $q(K_{n-k-\delta }\cup \overline{K_{\delta +k}})=2(n-k-\delta-1)$,
and we define a characteristic vector $x$,
where $x_i=1$ if $i\in Y\cup Z$ and $x_j=0$ if $j\in X$.
Obviously $x$ is the eigenvector corresponding to $q(K_{n-k-\delta-1}\cup \overline{K_{\delta+k+1}})$.
Then we get
\[ \langle Q(G)x,x \rangle-\langle Q(K_{n-k-\delta}\cup \overline{K_{\delta+k}})x,x\rangle
=\delta (\delta +k)-4|E'|\geq 0. \]
By Rayleigh's principle, we have $q(G)\geq 2(n-k-\delta-1).$
\end{proof}

\begin{lemma}\label{prop33}
For each graph $G\in \mathcal{B}_2(n,k,\delta)$, we have $q(G)> 2(n-k-\delta-1)-1.$
\end{lemma}

\begin{proof}
Let $x$ be the column vector defined in the proof of Lemma \ref{lem32}. We obtain that
\[  \langle Q(G)x,x \rangle-\langle Q(K_{n-\delta-k}\cup \overline{K_{k+\delta}})x,x\rangle
=\delta(\delta+k)-4|E'|\geq -4.\]
Similarly, we have $q(G)\geq 2(n-k-\delta-1)-\frac{4}{\|x\|^2}>2(n-k-\delta-1)-1.$
\end{proof}

Let $G$ be a graph of $\mathcal{B}_2(n,k,\delta)$ with the largest signless Laplacian spectral radius,
and furthermore we may assume that the induced graph $G[Y]$ contains the largest number of edges.
Let $h$ be the eigenvector corresponding to $q(G)$.
We assume further that $\max_{i\in V(G)}h_i=1$ by a rescaling.

\begin{lemma}\label{prop34}
Assume $G\in \mathcal{B}_2(n,k,\delta)$.
For each $i\in X$, we have $h_i\leq \frac{\delta}{q(G)-\delta}$.
\end{lemma}

\begin{proof}
Applying equation (\ref{eq1}) at $i \in X$, we have
\[  (q(G)-d(i))h_i=\sum_{j \in Y}h_j. \]
Since $d(i)=\delta$ and $\max_{i\in V(G)}h_{i}=1$, we have $(q(G)-\delta )h_i\leq \delta $,
the proof is completed.
\end{proof}

Next, we divide $Y,Z$ into two parts, respectively.
\begin{gather*}
Y_1=\{i\in Y: d(i)=n-1\},\quad Y_2=\{i\in Y: d(i)\leq n-2\},\\
Z_1=\{i\in Z: d(i)=n-\delta-k-1\},\quad Z_2=\{i\in Z: d(i)\leq n-k-\delta-2\}.
\end{gather*}
Before proving Lemma \ref{prop35}, let us make clear the following truth:
$Z_1\neq \emptyset $ as $n\geq (\delta^2 +k\delta +7\delta +6k +9)(\delta^2 +k\delta +1)$.
We already know the upper bound of $h_i$ for each $i\in X$
and ${\delta }/{(q(G)-\delta)}<1$ when $n>2\delta +k +3/2$.
Therefore we get $\max_{i\in V(G)}h_{i}=\max_{i\in Y\cup Z}h_{i}$.
To obtain the difference between $\max_{i\in V(G)}h_{i}$ and $\min_{i\in Y \cup Z}h_{i}$,
we need to prove the following several results.

\begin{lemma}\label{prop35}
Assume $G\in \mathcal{B}_2(n,k,\delta)$ as defined above.
\begin{enumerate}[(1)]

\item If $Y_2\neq \emptyset$, then $h_i> h_j $ for all $i\in Z_1$ and $j\in Y_2$.

\item  If $Z_2\neq \emptyset$, then $h_i>h_j $ for all $i\in Z_1$ and $j\in Z_2$.

\item If $Y_1\neq \emptyset$, then $h_i>h_j$ for any $i\in Y_1$ and $j\in Z_1$.
\end{enumerate}
\end{lemma}

\begin{proof}
(1) By contradiction, we assume that there are some $i\in Z_1$ and $j\in Y_2$ satisfying $h_i\leq h_j$.
Let $k\in Y$ and $jk\in E(G)$, define a new graph $G'\in \mathcal{B}_2(n,k,\delta)$
by removing an edge $ik$ and adding a new edge $jk$.
Since
\[ \langle Q(G')h,h\rangle-\langle Q(G)h,h\rangle=(h_j-h_i)(h_i+h_j+2h_k)\geq 0,\]
we get $q(G')\geq q(G)$ and the induced graph $G'[Y]$ has more edges than $G[Y]$,
which contradicts the choice of $G$. The proposition is proved.

(2) By Lemma \ref{lem21}, for each $i\in Z_1$ and $j\in Z_2$, we have
\[ (q(G)-d(i))(h_{i}-h_{j})
=(d(i)-d(j))h_{j}+\sum_{k\in N(i)\setminus N(j)}h_{k}-\sum_{l\in N(j)\setminus N(i)}h_{l}.\]
We notice that $N(j)\setminus \{i\}\subset N(i)\setminus \{j\}$, equivalently,
\begin{eqnarray*}
(q(G)-d(i)+1)(h_{i}-h_{j})=(d(i)-d(j))h_{j}+\sum_{k\in N(i)\setminus N[j]}h_{k}.
\end{eqnarray*}
Note that $d(i)>d(j),$ the proof is completed.

(3) If $Y_1\neq \emptyset$. For each $i\in Y_1$ and $j\in Z_1$, by Lemma \ref{lem21}, we have
\[ (q(G)-d(i))(h_{i}-h_{j})
=(d(i)-d(j))h_{j}+\sum_{k\in N(i)\setminus N(j)}h_{k}-\sum_{l\in N(j)\setminus N(i)}h_{l}.\]
Note that $N(j)\setminus \{i\}\subset N(i)\setminus \{j\}$, rearranging the last equation, we obtain
\begin{eqnarray}\label{equaion1}
(q(G)-d(i)+1)(h_{i}-h_{j})=(d(i)-d(j))h_{j}+\sum_{k\in N(i)\setminus N[j]}h_{k}.
\end{eqnarray}
Since $d(i)>d(j),$ the proof is completed.
\end{proof}

The key step for proving Lemma \ref{lem37} is to show Lemma \ref{prop36}.

\begin{lemma}\label{prop36}
Assume $G\in \mathcal{B}_2(n,k,\delta)$ as defined above. We have
\[ \max\limits_{i\in V(G)}h_i-\min\limits_{j\in Y\cup Z}h_j
< \frac{(\delta+k)(\delta +4)+4}{2(q(G)-n+1)}.\]
\end{lemma}

\begin{proof}
We distinguish it into two cases.

\textbf{Case 1}:
$Y_1\neq\emptyset.$ Notice that $\max_{i\in V(G)}h_i $ is attained by vertices from $Y_1$  by the above analysis. Hence $d(i)=n-1$, i.e., the vertex $i$ is adjacent to all other vertices in $G$.
\medskip

\textbf{Subcase 1.1}:
If $j\in Y_2,$ we have $N(i)\setminus N[j]=\{k:k\in Y\cup Z$ and $k\nsim j\}$.
Thus we have $d(i)-d(j)\leq \left\lfloor \frac{(\delta+k)\delta}{4} \right\rfloor +1$
and $|N(i)\setminus N[j]|\leq \left\lfloor \frac{(\delta+k)\delta}{4} \right\rfloor +1$.
Note that $N(j)\setminus \{i\}\subset N(i)\setminus \{j\}$, applying (4), we get
\begin{align*}
(q(G)-d(i)+1)(h_{i}-h_{j})
=(d(i)-d(j))h_{j}+\sum_{k\in N(i)\setminus N[j]}h_{k}
\leq  \frac{(\delta+k)\delta}{2} +2.
\end{align*}
Since $i\in Y_1$, then
\begin{align*}
h_{i}-h_{j} \leq  \frac{\frac{(\delta+k)\delta}{2}+2}{q(G)-(n-1)+1}
= \frac{(\delta+k)\delta +4}{2(q(G)-n+2)}
< \frac{(\delta+k)(\delta+4)+4}{2(q(G)-n+1)}.
\end{align*}

\textbf{Subcase 1.2}:
If $j\in Z_1$, we have $N(i)\setminus N[j]=X$ and $N(j)\setminus N[i]=\emptyset$.
Then $|N(i)\setminus N[j]|=\delta+k+1$.  Meanwhile, note that $d(i)-d(j)= \delta+k+1$. So
\begin{align*}
(q(G)-d(i)+1)(h_{i}-h_{j})
&=(d(i)-d(j))h_{j}+\sum_{k\in N(i)\setminus N[j]}h_{k}-\sum_{l\in X}h_{l} \\
&\leq  (d(i)-d(j))h_{j}+\sum_{k\in N(i)\setminus N[j]}h_{k} \\
&\leq   \delta+k+\delta+k. \\
&= 2(\delta+k).
\end{align*}
Since $i \in Y_1$, we easily have
\begin{align*}
h_{i}-h_{j}<\frac{2(\delta+k)}{q(G)-(n-1)+1}
= \frac{4(\delta+k)}{2(q(G)-n+2)}
< \frac{(\delta+k)(\delta+4)+4}{2(q(G)-n+1)}.
\end{align*}

\textbf{Subcase 1.3}:
If $j\in Z_2$, we have $N(i)\setminus N[j]=\{k:k\in X\cup Y\cup Z $ and $ k\nsim j\}$,
and $N(j)\setminus N[i]=\emptyset$. Thus we have
$|N(i)\setminus N[j]|\leq \delta+k+ \left\lfloor \frac{(\delta+k)\delta}{4} \right\rfloor+1$.
Then, note that $d(i)-d(j)\leq \left\lfloor \frac{(\delta+k)\delta}{4} \right\rfloor+1+\delta+k$.
Therefore, we get
\begin{align*}
(q(G)-d(i)+1)(h_{i}-h_{j})
   &=(d(i)-d(j))x_{j}+\sum_{k\in N(i)\setminus N[j]}x_{k}-\sum_{l\in X}x_{l} \\
   &\leq  (d(i)-d(j))x_{j}+\sum_{k\in N(i)\setminus N[j]}x_{k} \\
   &\leq   \frac{(\delta+k)\delta}{2}+2(\delta+k)+2. \\
   &= \frac{(\delta+k)(\delta +4)}{2}+2.
\end{align*}
Since $i \in Y_1,$ we easily have
\begin{align*}
h_{i}-h_{j} \leq \frac{\frac{(\delta+k)(\delta+4)}{2}+2}{q(G)-(n-1)+1}
= \frac{(\delta+k)(\delta+4)+4}{2(q(G)-n+2)}
< \frac{(\delta+k)(\delta+4)+4}{2(q(G)-n+1)}.
\end{align*}

\textbf{Case 2}:
$Y_1=\emptyset$. Notice that $\max_{i\in V(G)}h_i$ is attained by vertices from $Z_1$.
And for each $i \in Z_1,d(i)=n-\delta+k-1$,
hence the vertex $i$ is adjacent to all other vertices in $Y\cup Z$.

\textbf{Subcase 2.1}:
If $j\in Z_2$, we have $N(i)\setminus N[j]=\{k:k\in Y\cup Z$ and $k\nsim j\}$.
Thus we get $d(i)-d(j)\leq \left\lfloor \frac{(\delta+k)\delta}{4} \right\rfloor+1$
and $|N(i)\setminus N[j]|\leq \left\lfloor \frac{(\delta+k)\delta}{4} \right\rfloor+1$.
Note that $N(j)\setminus \{i\}\subset N(i)\setminus \{j\}$, applying (4), we get
\begin{align*}
(q(G)-d(i)+1)(h_{i}-h_{j})
=(d(i)-d(j))h_{j}+\sum_{k\in N(i)\setminus N[j]}h_{k}
\leq  \frac{(\delta+k)\delta}{2} +2.
\end{align*}
Since $i\in Z_1,$ we easily have
\begin{align*}
h_{i}-h_{j}
\leq \frac{\frac{(\delta+k)\delta}{2}+2}{q(G)-(n-\delta-k-1)+1}
= \frac{(\delta+k)\delta +4}{2(q(G)-n+\delta+k+2)}
< \frac{(\delta+k)(\delta +4)+4}{2(q(G)-n+1)}.
\end{align*}

\textbf{Subcase 1.2}:
If $j\in Y_2$, we have $N(i)\setminus N[j]=\{k: k\in Y\cup Z $ and $ k\nsim j\}$,
and $N(j)\setminus N[i]=X$.
Thus we have $|N(i)\setminus N[j]|\leq \left\lfloor \frac{(\delta+k)\delta }{4} \right\rfloor+1$.
We also observe that $|d(i)-d(j)|\leq \left\lfloor \frac{(\delta+k)\delta}{4} \right\rfloor+\delta+k+1$.
Similarly, we get
\begin{align*}
(q(G)-d(i)+1)(h_{i}-h_{j})
&=(d(i)-d(j))h_{j}+\sum_{k\in N(i)\setminus N[j]}h_{k}-\sum_{l\in X}h_{l} \\
&\leq  (d(i)-d(j))x_{j}+\sum_{k\in N(i)\setminus N[j]}x_{k} \\
&\leq  \frac{(\delta+k)\delta}{2}+\delta+k+1+1. \\
&= \frac{(\delta+k)(\delta+2)}{2}+2.
\end{align*}
Since $i \in Z_1,$ we easily have
\begin{align*}
h_{i}-h_{j} \leq \frac{\frac{(\delta+k)(\delta +2)}{2}+2}{q(G)-(n-\delta-k-2)+1}
= \frac{(\delta+k)(\delta +2)+4}{2(q(G)-n+\delta+k+2)}
< \frac{(\delta+k)(\delta+4)+4}{2(q(G)-n+1)}.
\end{align*}
The proof is completed.
\end{proof}

In order to prove Theorem \ref{thm14},  we also need another lemma.

\begin{lemma}\label{lem37}
Let $G$ be a graph of order $n\geq F(k,\delta)$ and  minimum degree $\delta(G) \geq \delta \geq 1$.
For each graph $G\in \mathcal{B}_2(n,k,\delta)$, we have $q(G)< 2(n-\delta-k-1)$.
\end{lemma}

\begin{proof}
We first assume $G\in  \mathcal{B}_2(n,k,\delta)$ such that $G$ has the largest signless Laplacian spectral radius and $G[Y]$ contains the largest number of edges.
Let $q(G)$ be the largest eigenvalue of $Q(G)$ and $h$ be the eigenvector corresponding to $q(G)$.
 Lemmas \ref{prop34} and \ref{prop36} together with
$|E'|=\left\lfloor \frac{\delta (\delta +k)}{4} \right\rfloor
\ge \frac{\delta (\delta +k)+1}{4}$ imply that
\begin{align*}
    & \langle Q(G)h,h\rangle-\langle Q(\overline{K_{\delta+k}}\cup K_{n-\delta-k})h,h\rangle \\
    &= \sum_{i\in X,j\in Y}(h_i+h_j)^2-\sum_{{i,j}\in E'}(h_i+h_j)^2 \\
    &\leq |X||Y|\left(\max_{i\in X}h_i +\max_{j\in Y}h_j\right)^2
          -|E'|\left(2\min_{j\in Y\cup Z}h_j\right)^2 \\
    &\leq \delta(k+\delta)\left(1+\frac{\delta}{q(G)-\delta}\right)^2
        -[\delta(\delta+k)+1]\left(1-\frac{(\delta+k)(\delta+4)+4}{2(q(G)-n+1)}\right)^2.
\end{align*}
Since $q(G)>2(n-\delta-k-1)-1$ by Lemma \ref{prop33},
then we can choose for our purpose that
$n\geq (\delta^2 +k\delta +7\delta +6k +9)(\delta^2 +k\delta +1)$,
to make sure the following holds.
\[  \langle Q(G)h,h\rangle-\langle Q(\overline{K_{\delta+k}}\cup K_{n-\delta-k})h,h\rangle <0.\]
According to Rayleigh's principle,
\[ \frac{\langle Q(\overline{K_{\delta+k}}\cup K_{n-\delta-k})h,h\rangle}{\langle h,h\rangle}
\leq q(\overline{K_{\delta+k}}\cup K_{n-\delta-k})=2(n-\delta-k-1).\]
Therefore, we have $q(G)=\frac{\langle Q(G)h,h\rangle}{\langle h,h\rangle}<2(n-\delta-k-1)$.
\end{proof}

Now we are ready to prove Theorem \ref{thm14}.

\begin{proof}
By Lemma \ref{lem22}, we have
\[ 2(n-\delta-k-1)\leq q(G) \leq \frac{2e(G)}{n-1}+n-2. \]
We obtain
\begin{align*}
    e(G)& \geq  \frac{(n-2\delta-2k)(n-1)}{2} \\
    &= \frac{1}{2}(n-k-\delta-1)(n-k-\delta-2)+(\delta +1)(k+\delta+1) \\
    &\quad +  n-\frac{1}{2}(k+\delta +1)(3\delta +k+2)-1 \\
    &> \frac{1}{2}(n-k-\delta-1)(n-k-\delta-2)+(\delta +1)(k+\delta+1),
    \end{align*}
the last inequality holds as $n\geq \frac{1}{2}(k+\delta +1)(3\delta +k+2)+1$.
By Lemma \ref{lem31}, $G$ is $k$-path coverable unless $G \subseteq {B}(n,k,\delta)$.
Together with Lemma \ref{lem32} and Lemma \ref{lem37}, the result follows.
\end{proof}

\end{document}